\documentclass[12pt,oneside,reqno]{amsart}
\usepackage{amsthm}
\usepackage{newtxmath}
\usepackage[utf8]{inputenc}
\usepackage{mathtools}
\usepackage{geometry}
\makeatletter
\numberwithin{equation}{section}
\numberwithin{figure}{section}

\usepackage{OldStandard}

\makeatother

\theoremstyle{plain}
\newtheorem*{prop*}{\protect\propositionname}
\newtheorem{lemma}{Lemma}
\providecommand{\propositionname}{Proposition}

\begin{document}
\title{A trigonometric approach to an identity by Ramanujan}
\author{C. Vignat }
\address{Department of Mathematics, Tulane University, cvignat@tulane.edu, Physics Department,
Université Paris Saclay}
\begin{abstract}
An identity by Ramanujan is expressed using polar coordinates, so that
its proof reduces to the verification of an elementary trigonometric identity.
This approach produces a few variations on Ramanujan's original identity.
\end{abstract}

\maketitle

\section{Introduction}

Entry 45 in Chapter 23 of Ramanujan's $4$th notebook \cite[p.102]{Berndt}
reads:
\begin{prop*}
Let $a,b,c$ and $d$ be any numbers such that $ad=bc.$ Then
\begin{equation}
64\left\{ \left(a+b+c\right)^{6}+\left(b+c+d\right)^{6}-\left(c+d+a\right)^{6}-\left(d+a+b\right)^{6}
+\left(a-d\right)^{6}-\left(b-c\right)^{6}\right\}\label{eq:Ramanujan}
\end{equation}
\[
\times\left\{ \left(a+b+c\right)^{10}+\left(b+c+d\right)^{10}-\left(c+d+a\right)^{10}-\left(d+a+b\right)^{10}+\left(a-d\right)^{10}-\left(b-c\right)^{10}\right\} 
\]
\[
=45\left\{ \left(a+b+c\right)^{8}+\left(b+c+d\right)^{8}-\left(c+d+a\right)^{8}-\left(d+a+b\right)^{8}+\left(a-d\right)^{8}-\left(b-c\right)^{8}\right\} ^{2}.
\]
\end{prop*}
This identity is qualified by B. Berndt as "one of the most fascinating finite identities we have ever seen". As far as we know, two proofs of this identity are available, one by Berndt
and Bhargava \cite{Berndt}, and the other one by Nanjundiah \cite{Nanjundiah}.
Berndt and Bhargava's proof is deduced from a clever factorization
of a polynomial reparameterization of the identity. Nanjundiah's proof
is the consequence of the properties of a degree 3 polynomial deduced
from the identity, and the identities satisfied by the powers of its
roots.

\section{A trigonometric approach}

We adopt here a different approach based on the following elementary trigonometric
result.
\begin{lemma}
\label{lemma1}
If $x_{1},y_{1},z_{1}$ and $x_{2},y_{2},z_{2}$ are real numbers
that satisfy the conditions
\begin{equation}
x_{i}+y_{i}+z_{i}=0,\,\,i=1,2\label{eq:centered}
\end{equation}
then there exists two radii $\rho_{1},\rho_{2}\ge0$ and two arguments
$\theta_{1}$ and $\theta_{2}$ such that, for $i=1,2,$
\[
\begin{cases}
x_{i}= & \rho_{i}\cos\theta_{i},\\
y_{i}= & \rho_{i}\cos\left(\theta_{i}-\frac{2\pi}{3}\right),\\
z_{i}= & \rho_{i}\cos\left(\theta_{i}+\frac{2\pi}{3}\right).
\end{cases}
\]
If additionally
\begin{equation}
\label{eq:xiyi}
x_{1}y_{1}+x_{1}z_{1}+y_{1}z_{1}=x_{2}y_{2}+x_{2}z_{2}+y_{2}z_{2},
\end{equation}
then
\[
\rho_{1}=\rho_{2}.
\]
\end{lemma}
Notice that both parameterizations 
\begin{equation}
\label{parameterization}  
\begin{cases}
x_{1}= & b+c+d\\
y_{1}= & -a-b-c\\
z_{1}= & a-d
\end{cases},\,\,\begin{cases}
x_{2}= & a+c+d\\
y_{2}= & -a-b-d\\
z_{2}= & b-c
\end{cases}
\end{equation}
satisfy  condition \eqref{eq:centered}. If additionally
$ad=bc$, they satisfy  condition \eqref{eq:xiyi}. Moreover, since identity \eqref{eq:Ramanujan}
is homogeneous in the variables $a,b,c,d$, we may assume $\rho=1.$ With $n\ge0,$ define the functions
\[
f_{n}\left(\theta\right)=\cos^{n}\theta+\cos^{n}\left(\theta-\frac{2\pi}{3}\right)+\cos^{n}\left(\theta+\frac{2\pi}{3}\right).
\]
The proof of Ramanujan's identity \eqref{eq:Ramanujan} is a consequence of the following Fourier series expansions, computed in Section 4.\begin{lemma}
The following linearizations hold
\[
f_{6}\left(\theta\right)=\frac{3}{32}\left(10+\cos\left(6\theta\right)\right),
\]
\[
f_{8}\left(\theta\right)=\frac{3}{128}\left(35+8\cos\left(6\theta\right)\right),
\]
\[
f_{10}\left(\theta\right)=\frac{27}{512}\left(14+5\cos\left(6\theta\right)\right).
\]
\end{lemma}
As a consequence,
\[
\left(f_{6}\left(\theta_{1}\right)-f_{6}\left(\theta_{2}\right)\right)\left(f_{10}\left(\theta_{1}\right)-f_{10}\left(\theta_{2}\right)\right)=\frac{405}{16384}\left(\cos\left(6\theta_{1}\right)-\cos\left(6\theta_{2}\right)\right)^{2}
\]
and
\[
\left(f_{8}\left(\theta_{1}\right)-f_{8}\left(\theta_{2}\right)\right)^{2}=\frac{9}{256}\left(\cos\left(6\theta_{1}\right)-\cos\left(6\theta_{2}\right)\right)^{2}
\]
and the proof of Ramanuan's identity \eqref{eq:Ramanujan} follows.

\section{Generalizations}

Noticing that 
\[
f_{3}\left(\theta_{1}\right)=\frac{3}{4}\cos\left(3\theta_{1}\right),f_{5}\left(\theta_{1}\right)=\frac{15}{16}\cos\left(3\theta_{1}\right),f_{7}\left(\theta_{1}\right)=\frac{63}{64}\cos\left(3\theta_{1}\right),
\]
the same arguments produce
\begin{prop*}
Without assuming $ad=bc,$
\[
25\left\{ \left(b+c+d\right)^{3}-\left(a+b+c\right)^{3}+\left(a-d\right)^{3}\right\} 
\left\{ \left(b+c+d\right)^{7}-\left(a+b+c\right)^{7}+\left(a-d\right)^{7}\right\} 
\]
\[
=21\left\{ \left(b+c+d\right)^{5}-\left(a+b+c\right)^{5}+\left(a-d\right)^{5}\right\} ^{2}
\]
which simplifies, replacing  $b+c$ by $b,$ to
\[
25\left\{ \left(b+d\right)^{3}-\left(a+b\right)^{3}+\left(a-d\right)^{3}\right\} 
\left\{ \left(b+d\right)^{7}-\left(a+b\right)^{7}+\left(a-d\right)^{7}\right\} 
\]
\[
=21\left\{ \left(b+d\right)^{5}-\left(a+b\right)^{5}+\left(a-d\right)^{5}\right\} ^{2}.
\]
\end{prop*}
This identity can be generalized to the case of two variables $\left(x_{i},y_{i},z_{i}\right),i=1,2,$
noticing that
\[
f_{3}\left(\theta_{1}\right)-f_{3}\left(\theta_{2}\right)=\frac{3}{4}\left[\cos\left(3\theta_{1}\right)-\cos\left(3\theta_{2}\right)\right]
\]
and accordingly for the functions $f_{5}$ and $f_{7},$ to obtain
\begin{prop*}
Assume $ad-bc=0,$ then
\begin{equation}
25\left\{ \left(b+c+d\right)^{3}-\left(a+b+c\right)^{3}-\left(a+c+d\right)^{3}+\left(a+b+d\right)^{3}+\left(a-d\right)^{3}-\left(b-c\right)^{3}\right\} 
\label{eq:Ramanujan-1}
\end{equation}
\[
\times\left\{ \left(b+c+d\right)^{7}-\left(a+b+c\right)^{7}-\left(a+c+d\right)^{7}+\left(a+b+d\right)^{7}+\left(a-d\right)^{7}-\left(b-c\right)^{7}\right\} 
\]
\[
=21\left\{ \left(b+c+d\right)^{5}-\left(a+b+c\right)^{5}-\left(a+c+d\right)^{5}+\left(a+b+d\right)^{5}+\left(a-d\right)^{5}-\left(b-c\right)^{5}\right\} ^{2}
\]
\end{prop*}

 Versions of these identities in the case $\rho_1\ne\rho_2$ can be produced, to the price of
symmetry. For example:
\begin{prop*}
Assuming $ad=bc,$ we have
\[
\hspace{-1cm}8\left(a^{2}+ab+b^{2}\right)\left(a^{2}+ac+c^{2}\right)
\left\{ \left(a+b+c\right)^{6}+\left(b+c+d\right)^{6}-\left(c+d+a\right)^{6}-\left(d+a+b\right)^{6}+\left(a-d\right)^{6}-\left(b-c\right)^{6}\right\} 
\]
\[
=3a^{2}\left\{ \left(a+b+c\right)^{8}+\left(b+c+d\right)^{8}-\left(c+d+a\right)^{8}-\left(d+a+b\right)^{8}+\left(a-d\right)^{8}-\left(b-c\right)^{8}\right\} 
\]
or equivalently
\[
4\left\{ \left(a+b+c\right)^{2}+\left(b+c+d\right)^{2}+\left(a-d\right)^{2}\right\} 
\]
\[
\times \left\{ \left(a+b+c\right)^{6}+\left(b+c+d\right)^{6}-\left(c+d+a\right)^{6}-\left(d+a+b\right)^{6}+\left(a-d\right)^{6}-\left(b-c\right)^{6}\right\} 
\]
\[
=3\left\{ \left(a+b+c\right)^{8}+\left(b+c+d\right)^{8}-\left(c+d+a\right)^{8}-\left(d+a+b\right)^{8}+\left(a-d\right)^{8}-\left(b-c\right)^{8}\right\} 
\]
\end{prop*}
The reason why parameterization \eqref{parameterization} was chosen by Ramanujan still eludes us.
\section{Proofs of Lemmas 1 and 2}
\subsection{Proof of Lemma \ref{lemma1}}
Denote $(x,y,z)$ one of both  triples  $(x_i,y_i,z_i),\,\,i=1,2$ and define
\[
r^{2}=x^{2}+y^{2}+z^{2}.
\]
Assuming $x+y+z=0$ produces
\[
\frac{r^{2}}{2}=\frac{x^{2}+y^{2}+z^{2}}{2}=-\left(xy+xz+yz\right)=-xy-z\left(x+y\right)
=-xy+\left(x+y\right)^{2}=x^{2}+y^{2}+xy
\]
so that
\[
\frac{r^{2}}{2}=X^{2}+Y^{2}
\]
with the new variables
\[
X=x+\frac{y}{2},\,\,Y=\frac{\sqrt{3}}{2}y.
\]
Define an argument $\alpha$ such that
\[
X=\frac{r}{\sqrt{2}}\cos\alpha,\,\,Y=\frac{r}{\sqrt{2}}\sin\alpha
\]
and deduce 
\[
x=\frac{r}{\sqrt{2}}\cos\alpha-\frac{1}{2}r\sqrt{\frac{2}{3}}\sin\alpha,\,\,y=r\sqrt{\frac{2}{3}}\sin\alpha
\]
that reduces to 
\[
x=r\sqrt{\frac{2}{3}}\cos\left(\alpha+\frac{\pi}{6}\right),y=r\sqrt{\frac{2}{3}}\cos\left(\alpha-\frac{\pi}{2}\right)
\]
and produces
\[
z=-x-y=r\sqrt{\frac{2}{3}}\cos\left(\alpha+\frac{5\pi}{6}\right).
\]
Defining
\[
\rho=r\sqrt{\frac{2}{3}},\,\,\theta=\alpha+\frac{\pi}{6}
\]
produces the result.

\subsection{A trigonometric identity}
\begin{prop*}
For two integers $p$ and $N$, we have the linearization 
\[
\frac{1}{N}\sum^{N-1}_{k=0}\cos^{2p}\left(\theta+\frac{2k\pi}{N}\right)=\sum^{p}_{\underset{2m\equiv0\mod N}{m=0}}a_{2m,p}\cos\left(2m\theta\right)
\]
with the Fourier coefficients
\[
a_{2m,p}=\begin{cases}
\frac{\binom{2p}{p+m}}{2^{2p-1}}, & m\ne0,\,\,m\le p\\
\frac{\binom{2p}{p}}{2^{2p}}, & m=0.
\end{cases}
\]
In the case of an odd power, 
\[
\frac{1}{N}\sum^{N-1}_{k=0}\cos^{2p+1}\left(\theta+\frac{2k\pi}{N}\right)=\sum^{p}_{\underset{2m+1\equiv0\mod N}{m=0}}a_{2m+1,p}\cos\left(\left(2m+1\right)\theta\right)
\]
with 
\[
a_{2m+1,p}=\frac{\binom{2p+1}{p-m}}{2^{2p+1}},\,\,0\le m\le p.
\]

\end{prop*}
\begin{proof}
The left-hand side is a periodic function of $\theta$ with period
$T=2\pi$: its Fourier coefficients are obtained as
\[
a_{n,p}=\frac{1}{2\pi}\int^{2\pi}_{0}\frac{1}{N}\sum^{N-1}_{k=0}\cos^{2p}\left(\theta+\frac{2k\pi}{N}\right)e^{-\imath n\theta}d\theta=\frac{1}{2\pi N}\sum^{N-1}_{k=0}\int^{2\pi}_{0}\cos^{2p}\left(\theta+\frac{2k\pi}{N}\right)e^{-\imath n\theta}d\theta.
\]
Using the change of variable $\eta=\theta+\frac{2k\pi}{N}$ and noticing
that, by periodicity, the domain of integration can be shifted back to $\left[0,2\pi\right],$
we deduce
\[
a_{n,p}=\frac{1}{2\pi N}\sum^{N-1}_{k=0}\int^{2\pi}_{0}\cos^{2p}\left(\eta\right)e^{-\imath n\left(\eta-\frac{2k\pi}{N}\right)}d\eta
=\frac{1}{2\pi}\int^{2\pi}_{0}\cos^{2p}\left(\eta\right)e^{-\imath n\eta}\left(\frac{1}{N}\sum^{N-1}_{k=0}e^{\imath2\pi\frac{kn}{N}}\right)d\eta.
\]
The inner sum is evaluated as the indicator function
\[
\sum^{N-1}_{k=0}e^{\imath2\pi\frac{kn}{N}}=\begin{cases}
0, & n\not\equiv0\mod N\\
1, & n\equiv0\mod N
\end{cases}
\]
In the case $n\equiv0\mod N,$ symmetry arguments show that the integral
vanishes for odd values of $m,$ and, for even values of $n=2m,$
\[
a_{2m,p}=\frac{1}{2\pi}\int^{2\pi}_{0}\cos^{2p}\left(\eta\right)e^{-\imath2m\eta}d\eta=\begin{cases}
\frac{\binom{2p}{p+m}}{2^{2p-1}}, & m\ne0\\
\frac{\binom{2p}{p}}{2^{2p}}, & m=0.
\end{cases}
\]
Notice that this result implies that $a_{2m,p}=0$ for $m>p.$
The odd power case is proved in the same way.
\end{proof}

\end{document}